\documentclass[a4paper]{article}
\usepackage{graphicx} 
\usepackage{amsmath, amsthm, amsfonts, amssymb, amscd, bm}
\usepackage{verbatim}
\usepackage{enumerate}
\usepackage{url}

\newtheorem{theorem}{Theorem}[section]
\newtheorem{corollary}{Corollary}[theorem]
\newtheorem{lemma}{Lemma}[section]
\newtheorem{definition}{Definition}[section]

\theoremstyle{remark}
\newtheorem{remark}{Remark}[section]

\def\F{\mathcal{F}}
\def\E{\mathcal{E}}
\def\H{\mathcal{H}}
\def\R{\mathbb{R}}
\def\SL{\mathcal{S\hspace{-0.03cm}L}}
\def\ext{\mathrm{ext}}

\begin{document}

\title{Sublinear Expectations and Martingales in discrete time}
\author{Samuel N. Cohen\footnote{University of Oxford, samuel.cohen@maths.ox.ac.uk. This work was begun while Samuel Cohen was visiting Shandong University from the University of Adelaide.} \\ Shaolin Ji\footnote{Shandong University, jsl@sdu.edu.cn.}\\ Shige Peng\footnote{Shandong University, peng@sdu.edu.cn.}}

\maketitle
\begin{abstract}
We give a theory of sublinear expectations and martingales in discrete time. Without assuming the existence of a dominating probability measure, we derive the extensions of classical results on uniform integrability, optional stopping of martingales, and martingale convergence. We also give a theory of BSDEs in the context of sublinear expectations and a finite-state space, including general existence and comparison results.


\end{abstract}

\section{Introduction}
The evaluation of risky outcomes in the presence of uncertainty is a difficult activity. Many problems, particularly in finance, revolve around the difficulties of assigning values to uncertain outcomes, when one may have very limited knowledge even of the probabilities associated with each outcome. In the presence of such `Knightian' uncertainty, many of our standard mathematical tools fail.

In this paper, we consider the natural generalisation of classical probability theory to such a situation. In this context, outcomes are evaluated, not using a single probability measure, but using the supremum over a range of measures, which may be mutually singular. This enables the consideration of such phenomena as volatility uncertainty in a financial market, and is closely connected to the theory of risk measures.

Following previous approaches, we axiomatically define nonlinear expectations as operators acting on a linear space of random variables. In this paper, we show how these nonlinear expectations can be extended to spaces with various completeness properties, including a theory of uniform integrability. We also give a theory of sublinear martingales in discrete time, including extensions of the classical results on optional stopping and martingale convergence.

In our context, we need to make few assumptions on the spaces under consideration, which gives our results a wide range of applicability. Our results apply to the case of $g$-expectations, as in \cite{Peng2005}, \cite{Cohen2010} and others, where the expectation is generated by the solutions to a BSDE, and our underlying space is the $L^2$ space under a classical probability measure. Our results also apply to the discrete-time theory of $G$-expectations, either through a direct discretisation of those considered in Peng \cite{Peng2010} and Soner, Touzi and Zhang \cite{Soner2010} in continuous time, or to the discrete-time version considered by Dolinsky, Nutz and Soner \cite{Dolinsky2011}.

We also give a summary of the theory of BSDEs in discrete-time finite-state systems, in the context of sublinear expectations. This restrictive setting is chosen to ensure that the filtrations considered have finite multiplicity, and hence that a finite-dimensional martingale representation result is possible. We give a representation result for nonlinear expectations in this context.

\section{Sublinear Expectations}
We give an axiomatic approach to the theory of sublinear expectations. We will not be overly concerned with the construction of a nonlinear expectation on a space, rather with the determination of the properties of such an expectation. In a finite-time finite-state system, it is known \cite{Cohen2009} that all nonlinear expectations can be constructed using the theory of Backward Stochastic Difference Equations. (This is not the case in general, unless the nonlinear expectation is assumed to have null sets equal to those of some linear expectation, see \cite{Cohen2009a}.)
                                                                                                                                                                                                                                                                     
Let $(\Omega, \F)$ be a measurable space and $\{\F_t\}_{t\in\mathbb{N}}$ be a discrete-time filtration on this space. Assume $\F=\F_\infty=\F_{\infty-}$ and $\F_0$ is trivial. Let $m\F_t$ denote the space of $\F_t$-measurable $\R\cup\{\pm\infty\}$-valued functions. Note that, in this context, the concepts of measurability, adaptedness, stopping times and the $\sigma$-algebra at a stopping time are identical to the classical case.

 \begin{definition} \label{defn:Hdefn}
  Let $\H$ be a linear space of $\F$-measurable $\R$-valued functions on
$\Omega$ containing the constants. We assume only that $X\in\H$ implies $|X|\in\H$ and $I_{A} X\in\H$ for any $A\in\F$, and define $\H_t := \H\cap m\F_t$.
 \end{definition}

\begin{definition} \label{defn:Edefn}
A family of maps $\E_t: \H\to \H_t$ is called a $\F_t$-consistent
nonlinear expectation if, for any $X,Y \in \H$, for all $s\leq t$,
\begin{enumerate}[(i)]
\item $X\geq Y$ implies $\E_t(X)\geq \E_t(Y)$.
\item $\E_s(Y) = \E_s(\E_t(Y))$
\item $\E_t(I_A Y) = I_A\E_t(Y)$ for all $A\in \F_t$.
\item $\E_t(Y)=Y$ for all $Y\in \H_t$.
\end{enumerate}
A nonlinear expectation is called sublinear if it also satisfies
\begin{enumerate}[(i)]
\setcounter{enumi}{4}
\item $\E_t(X+Y) \leq \E_t(X)+\E_t(Y)$
\item $\E_t(\lambda Y) = \lambda^+ \E_t(Y) + \lambda^-\E_t(-Y)$ for all $\lambda\in \H_t$ with $\lambda Y\in \H$.
\end{enumerate}
A nonlinear expectation is said to have the monotone continuity property (or Fatou property, see Lemma \ref{lem:Fatou}) if
\begin{enumerate}[(i)]
  \setcounter{enumi}{6}
 \item For any sequence $\{X_i\}$ in $\H$ such that
 $X_i(\omega)\downarrow 0$ for each $\omega$, we have $\E_0(X_i)\to 0$.
 \end{enumerate}

A $\F_t$-consistent sublinear expectation with the monotone continuity property will, for simplicity, be called a $\SL$-expectation.
As $\F_0$ is trivial, one can equate $\E_0$ with a map $\E:\H\to\mathbb{R}$, satisfying the above
properties.

\end{definition}

\begin{lemma}We can replace property (iii) and/or (vi) with
 \begin{enumerate}[(i')]
\setcounter{enumi}{2}
\item $\E_t(I_A X + I_{A^c}Y) = I_A \E_t(X) + I_{A^c}\E_t(Y)$ for all $A\in\mathcal{F}_t$
\setcounter{enumi}{5}
\item $\E_t(\lambda Y) = \lambda \E_t(Y)$ for all $0\leq\lambda\in \H_t$ with $\lambda Y\in \H$.
\end{enumerate}
without changing the definition of a sublinear expectation.
\end{lemma}
\begin{proof}
 Clearly (vi) implies (vi'). Given (vi'), it is easy to show that (iii) and (iii') are equivalent. Given (iii'), it is easy to show that (vi') implies (vi).
\end{proof}

\subsection{Basic Properties}
\begin{lemma}[From {\cite[Proposition 3.6]{Peng2010}}]
 Let $X, Y \in \H$ be such that $\E_t(Y)=-\E_t(-Y)$. Then for any
$\alpha\in\H_t$ such that $\alpha Y \in \H$,
\[\E_t(X+\alpha Y) = \E_t(X)+\alpha\E_t(Y).\]
In particular, $\E_t(X+\alpha) = \E_t(X)+\alpha$ for any $\alpha \in \H_t$, and
if $\E_t(Y) = \E_t(-Y)=0$, then $\E_t(X+ Y) = \E_t(X)$.
\end{lemma}
\begin{proof}
We have
\[\E_t(\alpha Y) = \alpha^+ \E_t(Y) + \alpha^{-}\E_t(-Y) = \alpha^+ \E_t(Y) -
\alpha^{-}\E_t(Y) = \alpha \E_t(Y)\]
and hence
\[\E_t(X+\alpha Y) \leq \E_t(X)+\E_t(\alpha Y)= \E_t(X)+\alpha\E_t(
Y)=\E_t(X)-\E_t(-\alpha Y)\leq \E_t(X+\alpha Y).\]
\end{proof}

\begin{lemma}[Jensen's inequality]\label{lem:jensen}
 For any convex function $\phi:\mathbb{R}\to \mathbb{R}$, any $t$, if $X$ and
$\phi(X)$
are both in $\H$, then
\[\E_t(\phi (X)) \geq \phi(\E_t(X)).\]
\end{lemma}
\begin{proof}
 First note that, $\E_t(aX)=a\E_t(X)$ for $a\geq 0$, and, by
sublinearity,
\[\E_t(aX) = |a|\E_t(-X) \geq -|a|\E_t(X) = a\E_t(X)\] for
$a<0$. As
$\phi$ is convex,
it has a representation $\phi(x) = \sup_i\{a_i x + b_i\}$ for some coefficients
$a_i, b_i$. Hence, by monotonicity
\[\E_t(\phi(X)) = \E_t(\sup_i\{a_i X + b_i\}) \geq \sup_i \E_t(a_i X +
b_i)\]
and, for each $i$,
\[\E_t(a_i X + b_i) =\E_t(a_i X) + b_i \geq a_i \E_t(X)+b_i.\]
The result follows.
\end{proof}

\begin{theorem}[See {\cite[Theorem 2.1]{Peng2010}}]
A $\SL$-expectation has a representation
\begin{equation}\label{eq:supremumrepresentation}
\mathcal{E}(Y) = \sup_{\theta\in\Theta} E_\theta[Y]
\end{equation}
where $\Theta$ is a collection of ($\sigma$-additive) probability measures on
$\Omega$.
\end{theorem}

Note that not all functionals with a representation of this form are $\SL$-expectations. This is
due to the requirements for recursivity in the definition of a $\F_t$-consistent
nonlinear expectation. Conditions such that the converse statement is true can
be found in Artzner et al. \cite{Artzner2007}.

\begin{definition} To use the language of capacity theory, we say that a
statement holds quasi-surely (q.s.) if it holds except on a set $N$ with
$\E(I_N)=0$, or, equivalently, if it holds $\theta$-a.s. for all
$\theta\in\Theta$. Such a set $N$ is called a polar set.
\end{definition}

The following results are natural generalisations of the classical ones.
\begin{lemma}[Semi-Strict Monotonicity] \label{lem:strictmonotone}
 Suppose $X\geq Y$ and $\E(X-Y)=0$. Then $X=Y$ q.s.
\end{lemma}
\begin{proof}
 Let $A_n=\{\omega: X >Y+n^{-1}\}$ for $n\in \mathbb{N}$. Then
\[0 = \E(X-Y) \geq \E(I_{A_n}(X-Y)) \geq n^{-1} \E(I_{A_n})\geq
0\]
and therefore $\E(I_{A_n})=0$, that is, $A_n$ is a polar set.
We can see that $\{\omega: X>Y\} = \bigcup_n A_n$ is the countable union
of $\theta$-null sets, hence $\{\omega: X>Y\}$ is also $\theta$-null, for all
$\theta\in \Theta.$ Hence $X=Y$ q.s.
\end{proof}

\begin{lemma}
 For any $X, Y \in \H$,
\[\E(X)-\E(Y) \leq \E(X-Y) \leq \E(X)+\E(-Y)\]
and hence
\[\E(X+Y) =\E(X-(-Y)) \geq \E(X)-\E(-Y).\]
In particular, note that $\E(Y) \geq -\E(-Y)$ for all $Y$.
\end{lemma}
\begin{proof}
 Simple application of sublinearity.
\end{proof}

\subsection{Extension of $\H$}

\begin{definition}
 For any pair $(\H, \E)$, where $\H$ satisfies Definition \ref{defn:Hdefn} and $\E$ is a $\SL$-expectation on $\H$, we can consistently extend our space to $(\H^{\ext}, \E^{\ext})$, where
\[\begin{split}
   \H^{\ext}&:=\left\{X\in m\F: \min\left\{E_\theta[X^+], E_\theta[X^-]\right\}<\infty \text{ for all }\theta\in\Theta\right\},\\
\E^{\ext}(X) &:= \sup_{\theta\in\Theta} E_\theta[X].
  \end{split}
\]

In particular, note that $\H\subseteq\H^{\ext}$ and $\E^{\ext}|_{\H} \equiv \E$, Jensen's inequality holds for $\E^{\ext}$, and $\E^{\ext}$ has the same representation as $\E$.
\end{definition}

\begin{remark}
The key purpose of this Definition is to allow us to consider the limits of sequences in our space $\H$, which will often lie in $\H^\ext$, but may not lie in $\H$. (For example, the limits of sequences of nonnegative random variables in $\H$ are guaranteed to lie in $\H^\ext$.) Using this extended definition, writing $\E^\ext(\lim X_n)$ has a meaning.

If $\H=\H^{\ext}$ \emph{ab initio}, then this poses no problem. However, in some applications it may be undesirable to directly define $\E$ over the larger space $\H^{\ext}$. For example, if $\H$ is the $L^2$ space under a probability measure and $\E$ is defined using the solutions of a classical BSDE. In these cases, this extension of $\E$ to $\E^{\ext}$ is convenient, and many of the properties of $\E$ carry through to $\E^{\ext}$ directly.

Note, however, that we do not require the use of dynamic expectation operators $\E^{\ext}_t$.
\end{remark}

The following two results depend on the extension of $\H$ to $\H^{\ext}$ only to guarantee that expectations can be taken of the required limits. When it can be assumed that all our sequences and their limits are in $\H$, the result can be taken with $\E^{\ext}$ replaced by $\E$. 

\begin{theorem}[Monotone Convergence Theorem]\label{thm:monconverge}
 Let $X_n$ be a nonnegative sequence in $\H^{\ext}$ increasing pointwise to $X$. Then
$\E^{\ext}(X_n)\uparrow \E^{\ext}(X)$.
\end{theorem}
\begin{proof}
 From the classical monotone convergence theorem, and the representation
(\ref{eq:supremumrepresentation}),
\[\E^{\ext}(X) = \sup_\theta E_\theta[X] = \sup_\theta \sup_n E_\theta[X_n] =
\sup_n \sup_\theta E_\theta[X_n] = \sup_n \E^{\ext}(X_n).\]
\end{proof}

\begin{lemma}[Fatou Lemma]\label{lem:Fatou}
 For any sequence  of nonnegative random variables $X_n\in\H^{\ext}$, we have $\E^{\ext}(\lim\inf X_n)\leq \lim\inf\E^{\ext}(X_n)$.
\end{lemma}
\begin{proof}
 Define $Y_n = \inf_{k\geq n} X_k$. Then $Y_n$ is a monotone sequence converging to $\lim\inf X_n$, and by the monotone convergence theorem
\[\E^{\ext}(\lim\inf_n X_n) = \lim_n \E^{\ext}(Y_n) = \lim_n \E^{\ext}(\inf_{k\geq n} X_k) \leq \lim_n \inf_{k\geq n} \E^{\ext}(X_k).\]
\end{proof}

\subsection{Conditioning on stopping times}
\begin{definition}
 For a stopping time $T$, we can define the expectation conditional on the
$\sigma$-algebra $\F_T$ by
\[\E_T(\cdot;\omega) := \E_{T(\omega)}(\cdot;\omega).\]
with the formal equivalence $\E_\infty(X) = X$.
\end{definition}

This nonstandard definition is necessary, as there is, in general, no way to
condition a nonlinear expectation on an arbitrary $\sigma$-algebra and retain
recursivity (except if the nonlinear expectation is of a very special type, see
\cite{Cohen2010a} for some related results in this area). However, as we are in
discrete time, the above definition does not introduce measurability problems, and we have the following
natural result.

\begin{lemma} \label{lem:boundedstoppingconsistent}
 For any bounded stopping time $T$, $\E(\E_T(\cdot))\equiv \E(\cdot)$.
\end{lemma}
\begin{proof}
Partition $\Omega$ into the
$\F$-measurable sets $A_n:=\{T=n\}$, for
$n=0,1,...$. Define $B_n = \bigcup_{k\geq n} A_k$. Then, for any $n$, any $X\in
H$,
\[\E_{T\wedge n}(X) = \sum_{k=0}^{n-1} I_{A_k} \E_k(X) + I_{B_n} \E_n(X).\]
Hence, as $B_n$ is $\F_{n-1}$-measurable,
\[\begin{split}
   \E_{n-1}(\E_{T\wedge n}(X)) &= \sum_{k=0}^{n-1} I_{A_k} \E_k(X) +
I_{B_n} \E_{n-1}(\E_n(X))\\
&=\sum_{k=0}^{n-2} I_{A_k} \E_k(X) +
I_{B_{n-1}} \E_{n-1}(X)\\
&= \E_{T\wedge(n-1)}(X)
  \end{split}
\]
Beginning with an upper bound for $T$ and iterating, we see that
\[\E_0(X) = \E_{T\wedge 0} (X) = \E_0(\E_{T\wedge1}(X)) =
\E_0(\E_1(\E_{T\wedge2}(X)))=... = \E_0(\E_T(X)).\]
\end{proof}

For $X$ satisfying a particular integrability property, we shall extend this result to unbounded stopping times, using properties of martingale convergence. See Theorem \ref{thm:recursivityatstoppingtimes}.

\section{Function spaces under $\E$}

\subsection{$L^p$ spaces}
As we no longer have a linear integral, we need to define the appropriate
analogue of the classical $L^p$ spaces.

\begin{definition}\label{defn:Lpdefn}
For $p\in[1,\infty[$, The map
\[\|\cdot\|_p:X\mapsto (\E(|X|^p))^{1/p}\]
forms a seminorm on $\H$. Similarly for $p=\infty$, where
\[\|\cdot\|_{\infty}:X\mapsto \inf_{x\in\mathbb{R}}\{x:\E(I_{\{X>x\}})=0\}.\]

We define the space $\mathcal{L}^p(\F)$ as the completion under $\|\cdot\|_p$ of the set
\[\{X\in \H:\|X\|_p<\infty\}\]
and then $L^p(\F)$ as the equivalence classes of $\mathcal{L}^p$ modulo equality
in $\|\cdot\|_p$.

Similarly, we can define $L^p(\F_t) = L^p(\F)\cap m\F_t$.
\end{definition}

\begin{remark}
 Note that, as we have defined $\mathcal{L}^p$ as the \emph{completion} of $\H$ under $\|\cdot\|_p$, we are permitted to take limits in $L^p$. Note also that $\mathcal{L}^p\subset\H^\ext$. The following lemma proves that $\E_t(X)$ is well defined for $X\in L^p$.
\end{remark}


\begin{lemma}\label{lem:convergenceofconditionalsinL1}
 Let $X_n$ be a sequence converging to $X$ in $L^p$, for $p\in [1,\infty]$. Then, for any $t$,
$\E_t(X_n)$ is convergent in $L^p$, and we can uniquely define $\E_t(X) := \lim \E_t(X_n)$. (If necessary, we can denote this extension $\E^{L^p}$, however we shall typically write $\E$ for simplicity.)
\end{lemma}
\begin{proof}
For $p<\infty$, note that by sublinearity and Jensen's inequality,
\[|\E_t(X_n)-\E_t(X_m)|^p\leq\E_t(|X_n-X_m|^p).\]
For $p=\infty$, note that
\[X_n \leq X_m + \|X_n-X_m\|_\infty \quad q.s.\]
Hence for any $p\in[1,\infty]$, we have $\|\E_t(X_n)-\E_t(X_m)\|_p \leq \|X_n-X_m\|_p$. In particular, $\E_t(X_n)$ is a Cauchy sequence in $L^p$. Hence it has a unique limit, which we call $\E^{L^p}_t(X)$. Comparing any two Cauchy sequences convergent to the same point, we see that $\E^{L^p}_t(X)$ is independent of the sequence chosen.
\end{proof}

\begin{lemma}\label{lem:consistentLpfromext}
 If $X\in L^p$ then $\E^\ext(X) = \E^{L^p}(X)$.
 \end{lemma}
 \begin{proof}
 By the classical dominated convergence theorem, we can see that $\E^{L^p}(X) = \sup_{\theta\in\Theta}E_\theta[X]$. The result follows.
\end{proof}

\begin{lemma}
 For any $1\leq p \leq q \leq \infty$, $L^q\subset L^p$.
\end{lemma}
\begin{proof}
 For $q=\infty$ the result is clear. For $q<\infty$, simply note that $x\mapsto|x|^{q/p}$ is convex, and the result follows from Jensen's inequality.
\end{proof}

\subsection{Uniform integrability}
To prove various convergence results, we require a notion of uniform
integrability under a nonlinear expectation. This notion is very similar to the
classical one, and hence, much of this section is based on \cite{Elliott1982}.

\begin{definition}
 Consider $K\subset L^1$. $K$ is said to be uniformly integrable (u.i.) if
$\E(I_{\{|X|\geq c\}} |X|)$ converges to $0$ uniformly in $X\in K$ as $c\to
\infty$.
\end{definition}

As is clear from Denis et al. \cite{Peng2010a}, in this context there is no guarantee that,
for a single random variable $X\in L^1$, the set $\{X\}$ is uniformly
integrable. For this reason, we define the following space.

\begin{definition}
 Let $L^p_b$ be the completion of the set of bounded functions $X\in \H$,
under the norm $\|\cdot\|_p$. Note that $L^p_b\subset L^p$.
\end{definition}

This set also has the following characterisation (see \cite[Prop. 18]{Peng2010a})
\begin{lemma}
For each $p\geq1$,
\[L^p_b = \{ X\in L^p : \lim_{n\to \infty} \E(|X|^p I_{\{|X|>n\}})=0\}.\]
\end{lemma}
\begin{proof}
For $p=\infty$, the result is trivial. For $p<\infty$, let $J_p = \{ X\in L^p : \lim_{n\to \infty} \E(|X|^p I_{\{|X|>n\}})=0\}$. For each $X\in J_p$, let $X_n = (X\wedge n)\vee(-n)$, which is bounded and in $\H$. We have
\[\E(|X-X_n|^p) \leq \E(|X|^p I_{\{|X|>n\}}) \to 0\]
and hence $X\in L^p_b$.

Conversely, for each $X\in L^p_b$, we can find a sequence of bounded random variables $Y_n$ in $\H$ such that $\E(|X-Y_n|^p)\to 0$. Let $y_n = \|Y_n\|_\infty$, and $X_n = (X\wedge y_n)\vee (-y_n)$. Since $|X-X_n|\leq |X-Y_n|$, we have $\E(|X-X_n|^p)\to 0$. This implies $\lim_n \E(|X-(X\wedge n)\vee(-n)|^p)=0$. We then have
\[\E(|X|^pI_{\{|X|>n\}}) \leq (1\vee 2^{p-1}) (\E((|X|-n)^pI_{\{|X|>n\}})+n^p\E(I_{\{|X|>n\}})),\]
but we know
\[\E((|X|-n)^pI_{\{|X|>n\}}) = \E(|X-(X\wedge n)\vee(-n)|^p) \to 0\]
and
\[(n/2)^p\E(I_{\{|X|>n\}}) \leq \E\left(\left(|X|-\frac{n}{2}\right)^pI_{\{|X|>{n/2}\}}\right) \to 0\]
from which we see $\E(|X|^pI_{\{|X|>n\}})\to 0$, that is, $X\in J_p$.
\end{proof}

From this characterisation, we can clearly see that any uniformly integrable set
 must lie in $L^1_b$. Unlike in \cite{Peng2010a}, we do not need to consider those $X\in L^p_b$ with a quasi-continuous version. For this reason, we do not need to assume that $\Omega$ posesses any topological structure.

\begin{lemma}
 For any $X\in L^1_b$, any $t$, $\E_t(X)\in L^1_b$.
\end{lemma}
\begin{proof}
 For any $X$ bounded, $\E_t(X)$ is also bounded. As $L^1_b$ is the
completion of the set of bounded functions under $L^1$-norm, from Lemma
\ref{lem:convergenceofconditionalsinL1}
the result is clear.
\end{proof}

The following lemma is easy to verify.
\begin{lemma}
 Any finite collection $K\subset L^1_b$ is uniformly
integrable.

The union of a finite number of uniformly integrable sets is also uniformly
integrable.
\end{lemma}

The following results are very closely based on the classical case.

\begin{theorem}\label{thm:unifintcharacter}
 Suppose $K$ is a subset of $L^1$. Then $K$ is uniformly integrable if and only
if both
\begin{enumerate}[(i)]
 \item $\{\E(|X|)\}_{X\in K}$ is bounded.
 \item For any $\epsilon>0$ there is a $\delta>0$ such that for all $A\in \F$
with $\E(I_A)\leq \delta$ we have $\E(I_A |X|)<\epsilon$ for all $X\in K$.
\end{enumerate}
\end{theorem}
\begin{proof}
 \emph{Necessity:} Define
\[ X^c(\omega) := \begin{cases} X(\omega) & |X(\omega)|\leq c\\ 0&
|X(\omega)|>c\end{cases}, \qquad\qquad X_c = X-X^c.\]
For any $A\in \F$, any $c>0$, we know $\E(I_A |X|) \leq \E(I_A (|X|\vee c)) \leq
c\E(I_A) + \E[|X_c|]$.

For any fixed $\epsilon>0$, if $K$ is uniformly integrable we can find a $c>0$
such that $\E[|X_c|]<\epsilon/2$ for all $X\in K$. Then $\E[|X|]\leq c +
\epsilon/2$ for all $X \in K$, establishing (i). For the same $c$, if $P(A)\leq
\delta=\epsilon/(2c)$ we have $\E(I_A |X|)<\epsilon$, proving (ii).

\emph{Sufficiency:} Fix $\epsilon>0$ and suppose (i) and (ii) hold. Then there
is a $\delta>0$ such that $\E[|X|]<\epsilon$ for all $A\in \F$ with $P(A)\leq
\delta$. Take
\[c = \delta^{-1} \sup_{X\in K} \E(|X|)<\infty.\]
For each $X\in K$, let $A_X=\{|X|\geq c\}$ so that, by Markov's inequality,
\[\E(I_{A_X}) = \sup_{\theta\in \Theta}\theta(\{|X|\geq c\})\leq c^{-1}
\sup_{\theta\in\Theta} E_\theta[|X|] = c^{-1}\E(|X|) \leq \delta.\]
Then
\[\E(I_{\{|X|\geq c\}} |X|) = \E(I_{A_X} |X|) < \epsilon\]
for all $X\in K$, so $K$ is uniformly integrable.
\end{proof}
\begin{corollary}
 Let $K$ be a subset of $L^1$. Suppose there is a positive function $\phi$
defined on $[0,\infty[$ such that $\lim_{t\to\infty} t^{-1} \phi(t) = \infty$
and $\sup_{X\in K} \E(\phi\circ |X|)<\infty$. Then $K$ is uniformly integrable.
\end{corollary}
\begin{proof}
 Write $\lambda = \sup_{X\in K} \E(\phi\circ |X|)$ and fix $\epsilon>0$. Put
$a=\epsilon^{-1}\lambda$ and choose $c$ sufficiently large that $t^{-1}\phi(t)
\geq a$ if $t\geq c$. Then on the set $\{|X|>c\}$ we have $|X|\leq a^{-1}
(\phi\circ |X|)$, so
\[\E(I_{\{|X|\geq c\}}|X|)\leq a^{-1} \E(I_{\{|X|\geq c\}}(\phi\circ|X|)) \leq
a^{-1} \E(\phi\circ |X|) \leq \epsilon,\]
and we see $K$ is uniformly integrable.
\end{proof}
\begin{corollary}
 For any $\epsilon>0$, $L^{1+\epsilon} \subset L^1_b$.
\end{corollary}
\begin{proof}
 Simply apply the previous corollary with $\phi(x) = x^{1+\epsilon}$ for any
$K=\{X\}$, $X\in L^{1+\epsilon}$.
\end{proof}

\begin{definition}
 A sequence $X_n\in \H^\ext$ will be said to converge in capacity to some
$X_\infty\in \H^\ext$ if, for any $\epsilon, \delta>0$, there exists an $N\in
\mathbb{N}$ such that
\[\E^\ext(I_{\{|X_m-X_\infty|>\epsilon\}})<\delta\]
for all $m\geq N$.
\end{definition}

\begin{lemma}
 A sequence $X_n$ which converges quasi-surely to $X_\infty$ also converges in
capacity. Conversely, a sequence $X_n$ which converges to $X_\infty$ in
capacity posesses a subsequence $X_{n_k}$ which converges quasi-surely to
$X_\infty$.
\end{lemma}
\begin{proof}
Fix $\epsilon>0$. Consider the sequence of measurable sets $A_n = \bigcup_{m\geq
n}\{|X_m-X|>\epsilon\}$. Let $A_\infty :=\bigcap_n A_n$. Then $I_{A_n}-I_{A_\infty}$ is a sequence decreasing to zero in $\H$, hence, by the monotone continuity property and sublinearity,
\[0\leq \E^\ext(I_{A_n})-\E^\ext(I_{A_\infty})\leq \E^\ext(I_{A_n}-I_{A_\infty})\downarrow 0,\]
and so $\lim_n \E^\ext(I_{A_n}) = \E^\ext(I_{A_\infty})$.

As $X_m$ converges except on a polar set $N$, on $N^c$ we know
$|X_m-X|>\epsilon$ for only finitely many $m$. Hence $A_\infty\cap N^c =
\emptyset$. Therefore, $\lim_n \E(I_{A_n}) = 0$, and so we can find the
required constant $K$ such that $\E(I_{A_n})< \delta$ for all $n\geq K$. Hence the first statment is
proven.

To show the converse, simply find a subsequence $n_k$ such
that $\E^\ext(I_{\{|X_{n_k}-X_\infty|\}})\leq 2^{-k}.$
Then, in the same way as for the Borel-Cantelli lemma,
\[\begin{split}\E^\ext(I_{\{X_{n_k} \not\to X_\infty\}})
&= \E^\ext(I_{\bigcap_N\bigcup_{k>N}\{|X_{n_k} - X_\infty|\geq k^{-1}\}})\\
&\leq \liminf_N \sum_{k>N} \E^\ext(I_{\{|X_{n_k}-X_\infty|\}}) =0\end{split}
\]
and so $X_{n_k}\to X$ q.s.
\end{proof}

We now prove the key convergence result for uniformly integrable sequences. A
similar result, in one direction (sufficiency) only can be found in Couso, Montes and Gil \cite{Couso2002}.

\begin{theorem}\label{thm:uiconvergence}
Suppose $X_n$ is a sequence in $L^1_b$, and $X\in \H^\ext$. Then the $X_n$ converge
in $L^1$ norm to $X$ if and only if the collection $\{X_n\}_{n\in\mathbb{N}}$
is uniformly integrable and the $X_n$ converge in capacity to $X$.

Furthermore, in this case, the collection $\{X_n\}_{n\in\mathbb{N}}
\cup \{X\}$ is also uniformly integrable and $X\in L^1_b$.
\end{theorem}
\begin{proof}
 \emph{Necessity:} Suppose the sequence $X_n$ converges in $L^1$ norm. That
$X_n$ converges in capacity follows from Markov's inequality. As $L^1_b$ is
complete under the $L^1$ norm, $X\in L^1_b$. For any $n$, we
have $\E(|X_n|)\leq \E(|X_n-X|)+\E(|X|)$, so $\E(|X_n|)$ is uniformly bounded.

For any $\epsilon>0$, let $N\in\mathbb{N}$ be such that $\E(|X_n-X|)< \epsilon/3$ for all
$n\geq N$. Then for any set $A\in \F$, this implies, for all $n\geq N$,
\[\E(I_A |X_n|) \leq \E(I_A|X|) +\E(|X_n-X|) <\E(I_A|X|)+\epsilon/3.\]
For any $n<N$, any $A\in \F$, we have
\[\E(I_A|X_n|) < \E(I_A|X|) + \E(I_A|X_n-X_N|) +\epsilon/3.\]
As $\{X\}\cup \{X_n-X_N\}_{\{n<N\}}$ is a finite collection in $L^1_b$, it is
u.i. Hence we can find a $\delta>0$ such that
$\E(I_A|X|)<\epsilon/3$ and $\E(I_A|X_n-X_N|)<\epsilon/3$ whenever
$\E(I_A)<\delta$. This implies that $\E(I_A |X_n|)<\epsilon$ whenever
$\E(I_A)<\delta$, that is, $\{X_n\}$ is u.i. As
$\{X_n\}_{n\in\mathbb{N}}\cup \{X\}$ is the union of two uniformly integrable
sets, it is also u.i., completing the proof.

\emph{Sufficiency:}
Suppose $\{X_n\}_{n\in \mathbb{N}}$ is uniformly integrable. As $X_n$ converges
in capacity, we can select a subsequence $n_k$ which converges quasi surely to
$X$. By Fatou's lemma, for any $A\in \F$,
\[\E(I_A|X|) = \E(\lim_k I_A|X_{n_k}|) \leq \lim\inf_k
\E(I_A|X_{n_k}|).\]
Hence, as $\{X_{n_k}\}$ is u.i.,  we see that $X\in L^1_b$.

Using the notation of Theorem \ref{thm:unifintcharacter},
\[\E(|X-X_n|) \leq \E(|X_n^c-X^c|) + \E(|(X_n)_c|) +
\E(|X_c|).\]
For any fixed $\epsilon>0$, as $\{X_n\}$ is u.i., $X\in L^1_b$, we can find
a $c>0$ such that $\E(|(X_n)_c|)<\epsilon/4$ and $\E(|X_c|)<\epsilon/4$.
Now note that, as $X_n^c \to X^c$ in
capacity, we can always find $N$ such that, for all
$n\geq N$, $\E(I_{\{|X_n-X|>\epsilon/4\}})<\epsilon/8c$. Hence, as
$|X_n^c-X^c|\leq 2c$, for all $n\geq N$ we have
\[\E(|X_n^c-X^c|)\leq \E(|X_n^c-X^c|\vee \frac{\epsilon}{4}) \leq
\frac{\epsilon}{4} + 2c \E(I_{\{|X_n-X|>\epsilon/4\}}) \leq
\frac{\epsilon}{2}.\]
Therefore $\E(|X-X_n|)<\epsilon$ for all $n\geq N$, as desired.
\end{proof}

\section{Symmetric processes and $\SL$-martingales.}
We now seek to use this framework to construct at theory of $\SL$-martingales (where $\SL$ refers to `sub-linear').
This is in many ways fundamentally simpler than in a continuous-time framework.

We begin by defining the key processes of interest.

\begin{definition}
A process $X$ is called a $\SL$-martingale if it satisfies
\[X_s = \mathcal{E}_s(X_t)\]
for all $s\leq t$, and $X_t \in L^1\cap m\F_t$ for all $t$. Similarly we define $\SL$-super-
and $\SL$-submartingales.
\end{definition}

\subsection {Optional stopping and Up/Downcrossing inequalities}
We now seek to reproduce the classical Optional stopping theorem, and the Up
and Down-crossing inequalities of the linear theory. The following definition
is standard.

\begin{definition}
 Let $X$ be an adapted, real-valued process. We define $M(\omega, X;[\alpha,
\beta])$ to be the number of upcrossings of the interval $[\alpha, \beta]$ made
by $X$, and similarly $D(\omega, X;[\alpha, \beta])$ for the number of
downcrossings.

For a stopping time $S$, we define $X^S$ to be the process $X$ stopped at $S$,
that is, $X^S_t(\omega) = X_{t\wedge S(\omega)}(\omega)$. As time is discrete, we know this is measurable, and for all finite $t$, $X^S_t\in\H$. If $S$ is bounded, then $X^S_\infty\in\H$ also.
\end{definition}
\begin{remark}
If $S$ is unbounded, we can only guarantee that $X^S_\infty \in \H^\ext$. This is because we have not assumed that for any countable partition $\{A_i\}\subset\F$ of $\Omega$, any $\{X_i\}\subset\H$ we have $\sum_i I_{A_i} X_i \in \H$. On the other hand, if we know that $X^S_t$ converges in $L^p$, then we can guarantee that $X^S_\infty\in L^p$.
\end{remark}

The Optional stopping and up/downcrossing results are
complicated by the fact that we cannot simply exchange $\SL$-super- and
$\SL$-sub-martingales by multiplication by -1. We do, however, have the following
Lemma.

\begin{lemma} \label{lem:negsuperissub}
 Let $(X_t)$ be a $\SL$-supermartingale. Then $(-X_t)$ is a $\SL$-submartingale.
\end{lemma}
\begin{proof}
By subadditivity, for any $s<t$ we have
\[-\E_s(-X_t)\leq \E_s(X_t)\leq X_s\]
and multiplication by $-1$ yields the result.
\end{proof}
(Note, however, that the converse result does not hold.) 

%
%
%

\begin{theorem}[Bounded optional stopping for submartingales]
 Let $X$ be a $\SL$-submartingale, and $S\leq T$ be bounded stopping times. Then
$X_S\leq \E_S(X_T)$.
\end{theorem}
\begin{proof}
 By the strict montonicity of $\E$ (Lemma \ref{lem:strictmonotone}), we
only need to show that, for $A=\{X_S>\E_S(X_T)\}\in \F_S$,
\[\E(I_A(X_S-\E_S(X_T)))=\E(I_A(X_S-X_T))\leq 0.\]

First assume that $S\leq T\leq S+1$, and let $K$ be an upper bound for $T$. Then
let $B_n = A\cap \{S=n\}\cap \{T>S\}$ and $C=A\cap \{S=T\}$. By construction,
$B_n, C\in \F_S$ and $A=(\cup_n B_n) \cup C$. We have $X_S-X_T = X_n-X_{n+1}$ on
$B_n$, therefore,
\[\E(I_A(X_S-X_T)) = \E(\sum_{n=0}^K I_{B_n} (X_n-X_{n+1})).\]
Hence, if $k=K-1$,
\[\E(I_A(X_S-X_T)) = \E(\E_k(I_{B_k} (X_k-X_{k+1}))+\sum_{n<k} I_{B_n}
(X_n-X_{n+1}))\]
As $\E_k(I_{B_k} (X_k-X_{k+1}))\leq 0$, we have
\[\E(I_A(X_S-X_T)) \leq \E(\sum_{n<k} I_{B_n}(X_n-X_{n+1})).\]

Iterating this argument with decreasing $k=K-2, K-3,...$, we see that
\[\E(I_A(X_S-X_T))\leq 0.\]

Hence the result is proven when $T\leq S+1$. In the general case, write $R_n =
T\wedge(S+n)$, for $n=0,1,..., K$, and as the $R_n$ are also stopping times,
$S=R_0$, $T=R_K$ and $R_n\leq R_{n+1}\leq R_n+1$, we see that
\[X_{R_k}\leq \E_{R_k}(X_{R_{k+1}})\leq \E_{R_k}(\E_{R_{k+1}}(X_{R_{k+2}}))\]
and a clear inductive argument gives the desired result $X_S\leq \E_S(X_T)$.
\end{proof}
\begin{theorem}
Let $X$ be a $\SL$-supermartingale, and $S\leq T$ be bounded stopping times.
Then $X_S\geq \E_S(X_T)$.
\end{theorem}
\begin{proof}
The inequality $X_S\geq \E_S(X_T)$ can be shown in precisely the same way as the
previous result, however reversing the direction of the inequality used.
\end{proof}
\begin{corollary}
 Let $X$ be a $\SL$-martingale. Then $X_S = \E_S(X_T)$.
\end{corollary}

We now derive the up/downcrossing inequalities for $\SL$-submartingales.

\begin{theorem}[Up/Downcrossing inequalities for $\SL$-submartingales]
Let $X$ be a $\SL$-submartingale, then for any bounded stopping time $S$,
\[\begin{split}
   \E(M(\omega, X^S;[\alpha, \beta])) &\leq (\beta-\alpha)^{-1}
(\E((X_S-\alpha)^+)-(X_0-\alpha)^+)\\
\E(D(\omega, X^S;[\alpha, \beta])) &\leq -(\beta-\alpha)^{-1}
\E(-(X_S-\beta)^+) \\
&\leq (\beta-\alpha)^{-1} \E((X_S-\beta)^+)
  \end{split}
\]
\end{theorem}
\begin{proof}
Define $Y_t = (X_t^S-\alpha)^+$, note that by Jensen's inequality and the
optional stopping theorem, $Y$ is a $\SL$-submartingale. Let
\[ M:=M(\omega, X^S; [\alpha, \beta]) \leq M(\omega, Y; ]0, \beta-\alpha[)=:M'\]
and similarly for $D$. As $M, M', D$ and $D'$ are all finite-integer-valued and $\F$-measurable, Definition \ref{defn:Hdefn} implies that they are in $\H$.

We first prove the upcrossing inequality. Define two sequences of stopping times $T$ and $T'$, where
\[\begin{split}
  T_1 &= 0\\
  T'_k(\omega) &= S\wedge \min\{n: n\geq T_k \text{ and }
Y_n(\omega)=0\}\qquad\qquad k=1,2,...\\
  T_k(\omega) &= S\wedge \min\{n: n>T'_{k-1} \text{ and } Y_n(\omega)\geq
\beta-\alpha\} \qquad k=2,3,...\\
  \end{split}
\]
Eventually, we reach a $p$ such that $T_p=T'_p=S$, as $S$ is bounded and we are
in discrete time. Hence
\[\sum_{k=1}^p(Y_{T_{k+1}}-Y_{T'_k})+ \sum_{k=1}^p(Y_{T'_k} -
Y_{T_k})=Y_S-Y_0\]
The first collection of sums each correspond to an upcrossing of the
interval $]0,\beta-\alpha[$ by $Y$. We know that there are exactly $M'$ of
these. Hence
\[M'(\beta-\alpha)\leq
Y_S-Y_0+\sum_{k=1}^p(Y_{T_k}-Y_{T'_k})\]
and so
\begin{equation}\label{eq:upcrossingdecomp}
 (\beta-\alpha)\E(M')\leq \E(Y_S)-Y_0+
\sum_{k=1}^p(\E(Y_{T_k}-Y_{T'_k}))
\end{equation}

By the optional stopping theorem  we know that for each $k$,
$Y_{T_k}\leq \E_{T_k}(Y_{T'_k})$, and so, by the consistency of the
nonlinear expectation at bounded stopping times (Lemma
\ref{lem:boundedstoppingconsistent})
\[\E(Y_{T_k}-Y_{T'_k}) = \E(Y_{T_k}-\E_{T_k}(Y_{T'_k}))\leq 0\]
and the final term of (\ref{eq:upcrossingdecomp}) is nonpositive. Replacing
$Y$ with $(X-\alpha)^+$, the result follows as $M\leq M'$.

Similarly, for the downcrossing inequality,  we have a series of stopping times
\[\begin{split}
  R'_1 &= S\wedge \min\{n: Y_n(\omega)\geq \beta-\alpha\}\\
  R_k(\omega) &= S\wedge \min\{n: n>R'_{k} \text{ and } Y_n(\omega)=0\}
\qquad\qquad\qquad  k=1,2,...\\
  R'_k(\omega) &= S\wedge \min\{n: n>R_{k-1} \text{ and }
Y_n(\omega)\geq \beta-\alpha\}\qquad\qquad k=2,3,...\\
  \end{split}
\]
with $R_p=R'_p=S$ for some sufficiently large $p$.
by the optional stopping
theorem and Lemma \ref{lem:boundedstoppingconsistent}, we see that
$\E(Y_{R'_k}-Y_{R_k}) \leq 0$, and so
\[\E(\sum_{k=1}^p(Y_{R'_k}-Y_{R_k}))\leq \sum_{k=1}^p \E(Y_{R'_k}-Y_{R_k})
\leq 0 .\]
Each nonzero term in this sum, except possibly the last, corresponds to
a descent of $Y$ to zero from a value of at least $\beta-\alpha$. There are
precisely
$D'$ such terms. The last
nonzero term has a value of at least
$-(Y_S-(\beta-\alpha))^+ = -(X_S-\beta)^+$, and hence,
\[0\geq \E((\beta-\alpha)D' - (X_S-\beta)^+)\geq
(\beta-\alpha)\E(D')+\E(-(X_S-\beta)^+)\]
and a simple rearrangement yields the result.
\end{proof}

\begin{corollary}
 Let $X$ be a $\SL$-supermartingale, then for any bounded stopping time $S$,
\[\begin{split}
   \E(M(\omega, X^S; [\alpha, \beta]))
      &\leq(\beta-\alpha)^{-1}\E((X_S-\alpha)^-) \\
   \E(D(\omega, X^S; [\alpha, \beta]))
      &\leq (\beta-\alpha)^{-1}(\E((X_S-\beta)^-)-(X_0-\beta)^-)
  \end{split}\]
\end{corollary}
This can be seen by applying the above results to the $\SL$-submartingale $(-X)$.

\begin{remark}
As in the classical case, we can see that these inequalitites will extend directly to the continuous time case, simply by applying the inequality on finite subsets of the rationals, and then using the monotone convergence theorem. Consequently, we can see that, in continuous time, any $\SL$-submartingale has a continuous version, using classical arguments. This result is also useful as an inequality on nonlinear expectations (for example, as given by BSDE solutions) in the classical case.
\end{remark}

\subsection{Martingale convergence}
We now give conditions under which $\SL$-martingales (and related processes)
converge. We shall focus our attention on $\SL$-submartingales, as, by Lemma
\ref{lem:negsuperissub}, this has direct implications for $\SL$-supermartingales
also.

\begin{theorem}
 Let $(X_n)_{n\in{\mathbb{N}}}$ be a $\SL$-submartingale with
$\sup_k\E(|X_k|)<\infty$. Then the
sequence $X_n$ converges quasi-surely to a random variable $X_\infty \in \H^\ext$ with $\E^\ext(|X_\infty|)<\infty$.
\end{theorem}
\begin{proof}
 Let $X^k$ denote the process $X$ stopped at the deterministic stopping time
$k$. Then, from the above downcrossing inequalities, for any real
$\alpha<\beta$ values we have
\[\begin{split}
\E(M(\omega, X^k; [\alpha, \beta])) &\leq (\beta-\alpha)^{-1}
[\E((X_k-\alpha)^+) - (X_0-\alpha)^+]\\
&\leq (\beta-\alpha)^{-1}[\E(X_k^+) +|\alpha|]\\
&\leq (\beta-\alpha)^{-1}[\E(|X_k|) +|\alpha|]\\
  \end{split}\]
As $M(\omega, X^k; [\alpha, \beta])$ is monotone increasing to $M(\omega,
X;[\alpha, \beta])$, by the Monotone convergence theorem,
\[\E(M(\omega,X;[\alpha, \beta])) = \sup_k\E(M(\omega,X^k;[\alpha, \beta]))
\leq (\beta-\alpha)^{-1}[\sup_k\E(|X_k|) +|\alpha|]<\infty\]
and hence $M(\omega,X;[\alpha, \beta])$ is quasi-surely finite. Hence, the
event
\[H_{\alpha, \beta} :=\{M(\omega,X;[\alpha, \beta])=\infty\} = \{\lim\sup
X_k>\beta\}\cap \{\lim\inf X_k < \alpha\}\]
is polar. Taking the countable union
\[H := \bigcup_{\alpha, \beta \in \mathbb{Q}, \alpha<\beta} H_{\alpha, \beta}
= \{\omega:\lim\sup X_k > \lim\inf X_k\}\]
we observe that this is again a polar set, and so that $X_k \to X_\infty$ q.s.
for some measurable function $X_\infty$.

As $|X_\infty|$ is nonnegative, measurable and there is a sequence in $\H$ converging to it q.s., $|X_\infty|\in\H^{\ext}$. By the Fatou property, as $\E^{\ext}(|X_k|) = \E(|X_k|)$ for all $k$,
\[\E^\ext(|X_\infty|) = \E^\ext(\lim|X_k|) \leq \lim\inf\E(|X_k|)\leq \sup_k
\E(|X_k|)<\infty.\]
It follows that $X_\infty\in\H^\ext$.
\end{proof}

\begin{remark}
 We would like to show that if $X_\infty\in \H^\ext$ and $\E^\ext(|X_\infty|)<\infty$, then $X_\infty \in L^1$. While this result is believable, it remains an open problem, due to the difficulty of constructing a sequence in $\H$ convergent to $X_\infty$ in $\|\cdot\|_1$. If $X$ is a uniformly integrable $\SL$-submartingale, we shall see that such a result is possible (Theorem \ref{thm:supermartconvergence}).
\end{remark}

\begin{corollary}
 Let $(X_n)_{n\in\mathbb{N}}$ be a $\SL$-supermartingale with
$\sup_k\E(X^-_k)<\infty$. Then the sequence $X_n$ converges quasi-surely to a
random variable $X_\infty \in \H^\ext$ with $\E^\ext(|X_\infty|)<\infty$.
\end{corollary}
\begin{proof}
 As $\E(|X_k|) \leq \E(X_k) + 2\E(X_k^-) \leq \E(X_0) + 2(X_k^-)$, we see
that $\sup_k\E(|X_k|)<\infty$. Then apply the above theorem to the
$\SL$-submartingale $(-X)$.
\end{proof}

\begin{theorem} \label{thm:supermartconvergence}
 Let $\{X_n\}_{n\in\mathbb{N}}$ be a uniformly integrable $\SL$-submartingale.
Then taking $X_\infty = \lim_n X_n$, the process $\{X_n\}_{n\in \mathbb{N}\cup
\{\infty\}}$ is also a uniformly integrable $\SL$-submartingale. In particular, this implies that $X_\infty \in L^1_b$.
\end{theorem}
\begin{proof}
 We know that $X_n\to X_\infty$ quasi-surely, and hence in capacity and, by
uniform integrability and Theorem \ref{thm:uiconvergence}, in $L^1$ norm . Uniform integrability of the extended process follows, and therefore $X_\infty \in L^1_b$.

We only need to verify that for all $t$, $X_t \leq \E_t(X_\infty)$. As $X_n\to X_\infty$ in $L^1$, we know $\E_t(X_n) \to \E_t(X_\infty)$ in $L^1$. By the submartingale property, $\E_t(X_n)$ is q.s. a
nondecreasing sequence as $n\to \infty$, with lower bound $X_t$. Hence, $X_t \leq \sup_{n>t}
\E_t(X_n) = \E_t(X_\infty)$ q.s. as desired.
\end{proof}

\begin{corollary}
 Let $\{X_n\}_{n\in\mathbb{N}}$ be a u.i. $\SL$-(super-)martingale. Then taking
$X_\infty = \lim_n X_n$, the process $\{X_n\}_{n\in \mathbb{N}\cup
\{\infty\}}$ is also a u.i. $\SL$-(super-)martingale.
\end{corollary}
\begin{proof}
 For $\SL$-supermartingales, this result follows exactly as for
$\SL$-submartingales. For $\SL$-martingales, note that a $\SL$-martingale is both
a $\SL$-submartingale and a $\SL$-supermartingale.
\end{proof}

\begin{theorem}[Optional stopping]\label{thm:optionalstoppingfull}
 Let $S,T$ be any stopping times $S\leq T$, and  $\{X_n\}_{n\in \mathbb{N}}$ be a
u.i. $\SL$-martingale. Defining $X_\infty = \lim_{n\to\infty} X_n$, we have $X_S=\E_S(X_T)$ q.s.
\end{theorem}
\begin{proof}
 As $X$ is a u.i. $\SL$-martingale, we know $X^T$ is also a u.i. $\SL$-martingale, and $X_t^T=\E_t(X^T_\infty)$ q.s. As $\mathbb{N}$ is countable, we then have $X_{S(\omega)}=\E_{S(\omega)}(X^T_\infty) = \E_{S(\omega)}(X_{T(\omega)})$ q.s. and the result is clear.
\end{proof}

\begin{theorem}
Let $X\in L^1_b$. Then $Y_t:=\E_t(X)$ is a u.i. $\SL$-martingale.
\end{theorem}
\begin{proof}
 Clearly $Y$ is a $\SL$-martingale. To show it is uniformly integrable, note that
by Jensen's inequality,
\[\E(I_{\{|Y_t|>c\}} |Y_t|) \leq \E(I_{\{|Y_t|>c\}} |X|)\]
and as $X$ is in $L^1_b$, this converges uniformly to $0$ as $c\to \infty$.
\end{proof}

We can now extend the recursivity of $\E$ to all stopping times, when $\E$ is
restricted to $L^1_b$.

\begin{theorem}\label{thm:recursivityatstoppingtimes}
 Let $S,T$ be any stopping times, $S\leq T$, and $X\in L^1_b$. Then
$\E_S(X)=\E_S(\E_T(X))$.
\end{theorem}
\begin{proof}
 Let $Y_t:=\E_t(X)$. Then $Y$ is a u.i. martingale and by
Theorem \ref{thm:optionalstoppingfull}, $Y_S=\E_S(Y_T)$, as desired.
\end{proof}

\subsection{Symmetry and Doob Decomposition}

\begin{definition}
A random variable $X\in\H$ (or $L^p$) will be called symmetric if
\[\mathcal{E}_t(X) = -\mathcal{E}_t(-X)\]
for all $t$.

A stochastic process $X$ will be called symmetric if each of its values $X_t$ is
a symmetric random variable.
\end{definition}

\begin{lemma}
A $\SL$-martingale $X$ is symmetric if and only if $-X$ is also a $\SL$-martingale,
that is, if
\begin{equation}\label{eq:symdefn1}
X_s=\mathcal{E}_s(X_t) = -\mathcal{E}_s(-X_t)
\end{equation}
for all $s<t$.
\end{lemma}
\begin{proof}
Simply multiply (\ref{eq:symdefn1}) through by $-1$, and note that this holds if
and only if $-X$ is a $\SL$-martingale.
\end{proof}

\begin{lemma}
This is equivalent to stating
\begin{equation}\label{eq:symdefn2}
X_{t-1} = -\mathcal{E}_{t-1}(-X_t)
\end{equation}
for all $t$.
\end{lemma}

\begin{proof}
Clearly (\ref{eq:symdefn1}) implies (\ref{eq:symdefn2}). To show the converse,
note that if $X$ is a $G$-martingale, then
$0=\mathcal{E}_{t-1}(X_t-X_{t-1})$. Hence, from (\ref{eq:symdefn2}),
\[\begin{split}
X_{t-2} &= -\mathcal{E}_{t-2}(-X_{t-1})=
-\mathcal{E}_{t-2}(-\mathcal{E}_{t-1}(X_{t}))\\&=
-\mathcal{E}_{t-2}(\mathcal{E}_{t-1}(-X_{t})) =
-\mathcal{E}_{t-2}(-X_{t}).\end{split}\]
The case for general $s<t$ follows by induction.
\end{proof}

We now give the analogue of the Doob decomposition theorem. Note that, in this
context, a process $Y$ is predictable if $Y_t\in m\F_{t-1}$ for all $t$.

\begin{theorem}\label{thm:compensatorexists}
For any adapted process $\{X_n\}_{n\in\mathbb{N}}\subset L^1$, there exists a unique
predictable process $\hat X$
such that $X-\hat X$ is a $\SL$-martingale and $\hat X_0=0$. $\hat X$ is called
the $\E$-\emph{compensator} of $X$.
\end{theorem}
\begin{proof}
Let $\hat X_0=0$. Then define inductively
\[\hat X_t = \mathcal{E}_{t-1}(X_t)-(X_{t-1} - \hat X_{t-1}).\]
It follows that
\[\mathcal{E}_{t-1}(X_t-\hat X_t) =
\mathcal{E}_{t-1}(X_t)-\mathcal{E}_{t-1}(X_t)+(X_{t-1} - \hat X_{t-1})= X_{t-1}
- \hat X_{t-1}.\]
That $X_t-\hat X_t$ is a $\SL$-martingale follows by recursion. The uniqueness of
$\hat X_t$, given the uniqueness of $\hat X_{t-1}$ is easy to verify, as the
addition of any nonzero $\mathcal{F}_{t-1}$ measurable quantity would destroy
the desired equality.
\end{proof}

\begin{lemma}\label{lem:compensatorsymmetry}
A process $X$ is symmetric if and only if $-\hat X$ is the $\E$-compensator of
$-X$.
\end{lemma}
\begin{proof}
As $X-\hat X$ is a $\SL$-martingale, we have
\[-X_{t-1} - (-\hat X_{t-1}) = -\mathcal{E}_{t-1}(X_t) - (-\hat X_{t})\]
and this is equal to $\mathcal{E}_{t-1}(-X_t) - (-\hat X_{t})$ for all $t$ if
and only if $X$ is symmetric. Uniqueness of the $\E$-compensator then completes
the argument.
\end{proof}

\begin{lemma}
An $L^1$ process $X$ is symmetric if and only if $X-\hat X$ is also symmetric.
\end{lemma}
\begin{proof}
As $X-\hat X$ is a $\SL$-martingale, it is symmetric if and only if $-X+\hat X$ is
a $\SL$-martingale. If so, it is clear that $-\hat X$ is the $\E$-compensator of $X$,
and so by Lemma \ref{lem:compensatorsymmetry}, the result is proven.
\end{proof}

\begin{lemma}
$X$ is a $\SL$-submartingale if and only if its $\E$-compensator is nondecreasing. $X$
is a $\SL$-supermartingale if and only if its $\E$-compensator is nonincreasing.
\end{lemma}
\begin{proof}
By construction
\[\hat X_t -\hat X_{t-1}= \mathcal{E}_{t-1}(X_t)-X_{t-1}\]
and the result is clear.
\end{proof}

\begin{theorem} \label{thm:martingaleintegralsaremartingales}
For any symmetric $\SL$-martingale $X$, any adapted process $Z$, the `integral'
process
\[Y_t = Y_0+\sum_{0\leq u<t} Z_u(X_{u+1}-X_u)\]
is a symmetric $\SL$-martingale.
\end{theorem}

\begin{proof}
We shall first show that it is a $\SL$-martingale. By construction,
\[ Y_{t+1}-Y_{t} = Z_t(X_{t+1}-X_t)\]
and so 
\[\mathcal{E}_t(Y_{t+1}-Y_{t}) = \mathcal{E}_t(Z_t(X_{t+1}-X_t)) =
Z_t^+\mathcal{E}_t(X_{t+1}-X_t)+Z_t^-\mathcal{E}_t(-X_{t+1}+X_t)\]
as $X$ is a symmetric $\SL$-martingale, we have
\[\mathcal{E}_t(X_{t+1}-X_t) =0= \mathcal{E}_t(-X_{t+1}+X_t)\]
and so 
\[Y_t=\mathcal{E}_t(Y_{t+1})\]
The result that $Y$ is a $\SL$-martingale follows by induction.

To show $Y$ is symmetric, we note that as $-X$ is also a symmetric
$\SL$-martingale, the above argument shows $-Y$ is also a $\SL$-martingale.
\end{proof}
\begin{corollary}\label{cor:symscalarmultiplication}
For any symmetric martingale (resp. random variable, stochastic process) $X$,
any constant $a$, the product $aX$ is symmetric.
\end{corollary}

\begin{lemma} \label{lem:symsumsaresymmartingales}
For any $\SL$-martingales $X,Y$ with $Y$ symmetric, the process $X+Y$ is a
$\SL$-martingale. Furthermore, $X+Y$ is symmetric if and only if $X$ is also
symmetric.
\end{lemma}

\begin{proof}
We know that, for any $s<t$, we have
\[\mathcal{E}_s(X_t+Y_t) \leq \mathcal{E}_s(X_t)+\mathcal{E}_s(Y_t)\]
and
\[\mathcal{E}_s(X_t+Y_t) \geq \mathcal{E}_s(X_t) - \mathcal{E}_s(Y_t) =
\mathcal{E}_s(X_t)+\mathcal{E}_s(Y_t)\]
and therefore
\[\mathcal{E}_s(X_t+Y_t)= \mathcal{E}_s(X_t)+\mathcal{E}_s(Y_t) = X_s+Y_s.\]
Repeating this argument shows $-X-Y$ is a $\SL$-martingale if and only if $-X$ is
also a $\SL$-martingale.
\end{proof}

\begin{lemma}\label{lem:symmetryisvectorspace}
The set of symmetric martingales (resp. random variables, processes) forms a
vector space.
\end{lemma}
\begin{proof}
This is a simple combination of Corollary \ref{cor:symscalarmultiplication} and
Lemma \ref{lem:symsumsaresymmartingales}.
\end{proof}

Recall that, in this context, a process $X$ is called predictable if $X_t$ is
$\mathcal{F}_{t-1}$-measurable for all $t$.

\section{Martingale Representation Theorem}
We seek to find a version of the martingale representation theorem in this setting. As in the classical case, we cannot do this for general discrete time filtrations, as the filtration has infinite multiplicity. We can, however, give a result for discrete-time, finite-state filtrations, which we consider here.

 For the infinite-state case, it is possible to consider BSDEs without a martingale representation theorem (leading to a theory similar to that in \cite{Cohen2009a}), or, if it can be assumed that $L^2(\F_T)$ is separable and admits a Schauder basis, then an infinite-dimensional approach (similar to that in \cite{Cohen2010a} in general classical probability spaces in continuous time) is also possible.

\begin{definition}
 Consider a discrete time process $(\mathcal{X}_t)_{t\in\mathbb{N}}$, taking
finitely many values. Without loss of generality, suppose that these values are
taken from the standard basis vectors in $\mathbb{R}^N$. Define $\mathbb{X}_t =
\sum_{u=1}^t \mathcal{X}_u$. We consider a probability space with filtration
generated by $\mathcal{X}$, that is, $\F_t = \sigma(\mathcal{X}_1,
\mathcal{X}_2,..., \mathcal{X}_t)$.
\end{definition}

\begin{theorem}[Semimartingale representation
theorem]\label{thm:semimartingalerepresentation}
Let $X$ be an adapted process. Then there exists a unique adapted process $Z$
such that 
\[\begin{split}
X_t &= X_0 + \sum_{0\leq u<t} Z_u^* \mathcal{X}_{u+1} = X_0 + \sum_{0\leq u<t}
Z_u^* (\mathbb{X}_{u+1}-\mathbb{X}_{u})\\
\end{split}\]
\end{theorem}

\begin{proof}
By application of the Doob-Dynkin Lemma, (as in \cite{Cohen2008c}), there exists
a unique $\mathcal{F}_{t}$ measurable random vector $Z_t$ such that 
\[X_{t+1}-X_{t} = Z_t^*\mathcal{X}_{t+1}=
Z_t^*(\mathbb{X}_{t+1}-\mathbb{X}_{t}).\]
A simple telescoping sum establishes the result.
\end{proof}

We now proceed to give a type of martingale representation. In
many situations, for the representation of $\SL$-martingales, we wish
to decompose a martingale in terms of a symmetric component and a remainder,
where the symmetric component is the `integral' of a given process. This
motivates the following result, an application of which is found in Section
\ref{sec:TrinomialExample}.

\begin{theorem}\label{thm:linearmartrepthm}
Consider an invertible adapted bounded linear map $\Phi:\mathbb{R}^N\to\mathbb{R}^N$,
represented by premultiplication by the matrix $[\mathbf{1}|(\phi^s_u)^*|(\phi'_u)^*]^*$ where $\mathbf{1}$ is a column of 1's, $\phi^s_u\in\mathbb{R}^M\times\mathbb{R}^N$ for some $0\leq M\leq N-1$,
$\phi'_u\in\mathbb{R}^{N-M-1}\times\mathbb{R}^N$ and
\[\sum_{0\leq u< t}\phi^s_{u}(\mathbb{X}_{u+1}-\mathbb{X}_u)\]
 is a symmetric martingale. Then, for any $\SL$-martingale $X$, there exist unique
adapted processes $Z^{(s)}\in\mathbb{R}^M, Z'\in\mathbb{R}^{N-M-1}$ such that
\[\begin{split}
X_t &= X_0 + \sum_{0\leq
u<t}(Z_u^{(s)})^*\phi^s_u(\mathbb{X}_{u+1}-\mathbb{X}_{u}) \\
&\qquad + \sum_{0\leq u<t}(Z_u')^*\phi'_u(\mathbb{X}_{u+1}-\mathbb{X}_{u}) -
\sum_{0\leq u<t}G_u(Z_u')\end{split}\]
where
$G_{t}(Z_t')=\mathcal{E}_{t}((Z_t')^*\phi'_t(\mathbb{X}_{t+1}-\mathbb{X}_{t}))$.
\end{theorem}

\begin{proof}
By Theorem \ref{thm:semimartingalerepresentation}, we have a unique
representation of the form
\[\begin{split}
X_t &= X_0 + \sum_{0\leq u<t} Z_u (\mathbb{X}_{u+1}-\mathbb{X}_u)\\
&= X_0 + \sum_{0\leq u<t} Z_u^*
\Phi_u^{-1}\Phi_u(\mathbb{X}_{u+1}-\mathbb{X}_u)\\
\end{split}\]
Let 
\[Z_u^*
\Phi_u^{-1}=:[(Z^{(\wedge)}_u)^*|(Z^{(s)}_u)^*|(Z'_u)^*]\in\mathbb{R}\times
\mathbb{R}^M\times\mathbb{R}^{N-M-1}.\]
Then, as $\mathbf{1}^*(\mathbb{X}_{u+1}-\mathbb{X}_u)=\mathbf{1}^*\mathcal{X}_u
= 1$, we clearly have
\[\begin{split}
X_t &= X_0 + \sum_{0\leq u<t}Z^{(\wedge)}_u + \sum_{0\leq
u<t}(Z_u^{(s)})^*\phi^s_u(\mathbb{X}_{u+1}-\mathbb{X}_{u}) \\
&\qquad + \sum_{0\leq
u<t}(Z_u')^*\phi'_u(\mathbb{X}_{u+1}-\mathbb{X}_{u})\end{split}\]

Taking a $\E_{t-1}$ expectation yields
\[\E_{t-1}(X_t)= X_{t-1}= X_{t-1} + Z^{(\wedge)}_{t-1} +
\E_{t-1}\left((Z_{t-1}')^*\phi'_{t-1}(\mathbb{X}_{t}-\mathbb{X}_{t-1}
)\right)\]
and so substitution with $G_{t-1}(Z'_{t-1}) =
\E_{t-1}\left((Z_{t-1}')^*\phi'_{t-1}(\mathbb{X}_{t}-\mathbb{X}_{t-1}
)\right)$ gives the desired result.
\end{proof}

\begin{remark}
 It is worth noting that the function $G$ defined here plays a very similar role as in the theory of $G$-expectations. This will become clear in the following example, where $G$ precisely describes the volatility uncertainty of our processes.
\end{remark}

\subsection{A Trinomial Example}\label{sec:TrinomialExample}
To illustrate the above results, we consider the following basic trinomial
model. Let $\mathcal{X}_t$ take values from the standard basis vectors in
$\mathbb{R}^3$, $\mathbb{X}_t = \sum_{0<u\leq t} \mathcal{X}_u$. We consider the
space of test probabilities which assign, at time $t-1$, a probability $p_t$ to
the events $\mathcal{X}_t=e_1$ and $\mathcal{X}_t=e_3$, and probability $1-2p_t$
to the event $\mathcal{X}_t=e_2$. The conditional nonlinear expectation is then
obtained by maximising over the set $p_t\in[\epsilon,1/2-\epsilon]$ for some
fixed $\epsilon\in[0,1/4]$. 

More formally, we can define the nonlinear expectation, for any random variable
$X=f(\mathcal{X}_1, \mathcal{X}_2, ..., \mathcal{X}_T)$, by the recursion
\[\E_{t-1}(X) = \max_{p_{t}\in[\epsilon,
1/2-\epsilon]}\left\{\begin{split}
p_{t} \times&\E_{t}(f(\mathcal{X}_1, ...,\mathcal{X}_{t-1}, e_1,\mathcal{X}_{t+1}, ...\mathcal{X}_T))\\
+ (1-2p_{t})\times&\E_{t}(f(\mathcal{X}_1, ...,\mathcal{X}_{t-1}, e_2,\mathcal{X}_{t+1}, ...\mathcal{X}_T))\\
+ p_{t} \times &\E_{t}(f(\mathcal{X}_1, ...,\mathcal{X}_{t-1}, e_3,\mathcal{X}_{t+1}, ...\mathcal{X}_T))
\end{split}\right\}\]
with terminal condition $\E_T(X)=X$.

If we now consider the process $B_t:=\sum_{0<u\leq t}[1,0,-1]\mathcal{X}_u$, we
can see that $B_t$ is the natural nonlinear version of a trinomial random walk,
and satisfies
\[\E_{t-1}(B_t)=-\E_{t-1}(-B_{t})=B_{t-1} \text{ for all
}t.\]
We also consider the process
\[[B]_t:= \sum_{0<u\leq t} (B_u-B_{u-1})^2 = \sum_{0<u\leq t} [1,0,1]
\mathcal{X}_u.\]

If $\phi^s_u:=[1,0,-1]$ and $\phi'_u:=[1,0,1]$, it is
clear therefore that 
\[\Phi_u = \left[\begin{array}{ccc}
1&1&1\\1&0&-1\\1&0&1\end{array}\right]\]
is invertible, and that 
\[B_t = \sum_{0<u\leq t}\phi^s_u\mathcal{X}_u= \sum_{0\leq
u<t}\phi^s_{u}(\mathbb{X}_{u+1}-\mathbb{X}_u)\]
is a symmetric $\SL$-martingale. Hence, by Theorem \ref{thm:linearmartrepthm}, for
any $\SL$-martingale $X$ in this space, we can find unique scalar processes
$Z^{(s)}, Z'$ such that
\[\begin{split}
X_t &= X_0 + \sum_{0\leq u<t}Z_u^{(s)}\phi^s_u(\mathbb{X}_{u+1}-\mathbb{X}_{u})
\\
&\qquad + \sum_{0\leq u<t}Z_u'\phi'_u(\mathbb{X}_{u+1}-\mathbb{X}_{u}) -
\sum_{0\leq u<t}G_u(Z_u')\\
&=X_0 + \sum_{0\leq u<t}Z_u^{(s)}(B_{u+1}-B_t) \\
&\qquad + \sum_{0\leq u<t}Z_u'([B]_{u+1} - [B]_u) - \sum_{0\leq
u<t}G_u(Z_u')
\end{split}\]
where $G_{t}(Z_t')=\mathcal{E}_{t}(Z_t'([B]_{u+1} - [B]_u))$.

\section{$\SL$-BSDEs}
We now consider the theory of $\SL$-BSDEs. These are the generalisation of BSDEs
(Backward Stochastic Difference Equations), as studied in \cite{Cohen2008c}, to
the context of $\SL$-martingales.

\begin{definition}
For a given map $\Phi$ as in Theorem \ref{thm:linearmartrepthm}, define the
processes 
\[\begin{split}
M_{t+1} &:= \phi_t^{(s)}(\mathbb{X}_{t+1}-\mathbb{X}_t)\\
N_{t+1} &:= \phi_t'(\mathbb{X}_{t+1}-\mathbb{X}_t)
\end{split}\]
For a given finite, deterministic time $T$, a $\SL$-BSDE is an equation of the form
\[Y_{t+1} = Y_t - F(\omega, t, Y_t, Z_t, Z_t') + Z_t M_{t+1} + Z_t' N_{t+1} -
G_t(Z_t')\]
with terminal value $Y_{T} = Q$ for $Q\in L^1(\F_T)$. Here 
\[\begin{split}
G_t(z) &:= \mathcal{E}_t(z N_{t+1}),\\
F&:\Omega\times \{0,1,...,T-1\}\times \mathbb{R}^K\times \mathbb{R}^{K\times
M}\times \mathbb{R}^{K\times(N-M-1)}\to\mathbb{R}^K
\end{split}\]
and $F(\cdot, t, y, z, z')\in L^1(\F_t)$ for all $t$. A solution to a $\SL$-BSDE is a
triple $(Y, Z, Z')$ of adapted processes (of appropriate dimension).
\end{definition}

For notational simplicity, we shall take the parameter $\omega$ of $F$ as implicit.

\begin{theorem}
Let $F$ be such that, for any $t, z, z'$, the map $f_{t, z, z'}:y\mapsto
y-F(t, y, z, z')$ is q.s. a bijection $\mathbb{R}^K\to\mathbb{R}^K$. Then a
BSDE with driver $F$ has a unique solution for any $Q\in L^1(\F_T)$.
\end{theorem}
\begin{proof}
Clearly there is a unique solution $Y_T=Q$ at time $T$. Suppose there is a
solution at time $t+1$. Then, there is a unique representation of the
$\SL$-martingale difference
\begin{equation}\label{eq:BSDEmartterm}
Y_{t+1}-\mathcal{E}_t(Y_{t+1})= Z_t M_{t+1} + Z'_t N_{t+1} - G(Z'_t),
\end{equation}
by Theorem \ref{thm:linearmartrepthm}.

Taking this pair $Z_t, Z'_t$, we now consider the equation
\begin{equation}\label{eq:BSDEtrendterm}
\mathcal{E}_t(Y_{t+1}) = Y_t - F(t, Y_t, Z_t, Z_t')=f_{t, Z_t,
Z_t'}(Y_t).
\end{equation}
As $f_{t, Z_t, Z_t'}$ is an adapted bijection, we obtain a unique solution $Y_t
= f^{-1}_{t, Z_t, Z_t'}(\mathcal{E}_t(Y_{t+1}))$.

The addition of (\ref{eq:BSDEmartterm}) and (\ref{eq:BSDEtrendterm}) gives a
solution triple $(Y_t, Z_t, Z_t')$ with the desired dynamics. It is also clear
that if the dynamics hold, both (\ref{eq:BSDEmartterm}) and
(\ref{eq:BSDEtrendterm}) must hold, and so the solution is unique. Finally,
backwards induction defines this solution for all $t$.
\end{proof}

We now restrict ourselves to the scalar case ($K=1$).

\begin{theorem}[Comparison Theorem]\label{thm:compthm}
Consider two BSDEs with terminal conditions $Q,\bar Q$, drivers $F, \bar F$ and solutions $(Y,Z,Z'), (\bar Y, \bar Z, \bar Z')$. Suppose that
\begin{enumerate}[(i)]
 \item $Q\geq \bar Q$ q.s.
 \item $F(t, y, z, z') \geq \bar F(t, y ,z, z')$ q.s. for all $t,y,z$ and $z'$
 \item $y\mapsto y-F(t, y, z,z')$ is strictly increasing, as a function of $y$, for all $t,z, z'$.
 \item For all $t, y$, all $(z,z')\neq (\bar z, \bar z')$,
\[\begin{split}&F(t, y, z, z')-F(t, y, \bar z, \bar z') + G_t(z') - G_t(\bar z')\\
&\qquad > \min_{\F_t}\{(z-\bar z) M_{t+1} + (z'-\bar z')N_{t+1}\},\qquad\qquad q.s.
\end{split}\] where $\min_{\F_t}$ denotes the essential minimum conditional on $\F_t$, (that is, the minimum on all non-polar paths of the state-tree extending from the current node).
\end{enumerate}
Then $Y_t \geq \bar Y_t$ for all $t$. Furthermore, we have the strict comparison: if for some $A\in \F_t$, $Y_t=\bar Y_t$ on $A$; then $Q=\bar Q$ on $A$.
\end{theorem}
\begin{proof}
Suppose $Y_{t+1}\geq \bar Y_{t+1}$. Then write
\[\begin{split}
   &Y_t-\bar Y_t- F(t, Y_t, \bar Z_t, \bar Z'_t) + F(t, \bar Y_t, \bar Z_t, \bar Z'_t) \\
&= (Y_{t+1} -\bar Y_{t+1})+ (F(t, \bar Y_t, \bar Z_t, \bar Z'_t) - \bar F(t, \bar Y_t, \bar Z_t, \bar Z'_t))\\
&\quad+ F(t, Y_t, Z_t, Z'_t) - F(t, Y_t, \bar Z_t, \bar Z'_t) + G_t(Z_t) - G_t(\bar Z'_t)\\
&\quad - (Z_t-\bar Z_t)M_{t+1} - (Z'_t - \bar Z'_t) N_{t+1}\\
&\geq F(t, Y_t, Z_t, Z'_t) - F(t, Y_t, \bar Z_t, \bar Z'_t) + G_t(Z_t) - G_t(\bar Z'_t)\\
&\quad - (Z_t-\bar Z_t)M_{t+1} - (Z'_t - \bar Z'_t) N_{t+1}.
  \end{split}\]
As this must hold up to a polar set, it holds under taking the essential maximum conditional on $\F_t$, that is,
\begin{equation}\label{eq:compthmbreakup}
 \begin{split}
   &Y_t-\bar Y_t- F(t, Y_t, \bar Z_t, \bar Z'_t) + F(t, \bar Y_t, \bar Z_t, \bar Z'_t) \\
&\geq F(t, Y_t, Z_t, Z'_t) - F(t, Y_t, \bar Z_t, \bar Z'_t) + G_t(Z_t) - G_t(\bar Z'_t)\\
&\quad +\max_{\F_t}\{ - (Z_t-\bar Z_t)M_{t+1} - (Z'_t - \bar Z'_t) N_{t+1}\}\\
&\geq 0
  \end{split}
\end{equation}
where the final inequality is by assumption (iv). As we then have
\[Y_t - F(t, Y_t, \bar Z_t, \bar Z'_t)\geq \bar Y_t - F(t, \bar Y_t, \bar Z_t, \bar Z'_t)\]
by assumption (iii), $Y_t\geq \bar Y_t$. Backwards induction then yields the general result.

To show the strict comparison, note that if $Y_t=Y'_t$ on some $A\in\F_t$, then the left hand side of (\ref{eq:compthmbreakup}) is zero. Hence, by assumption (iv), $(Z_t, Z'_t) = (\bar Z_t, \bar Z'_t)$ on $A$. We then have that, on $A$,
\[Y_{t+1} - \bar Y_{t+1} = -F(t,Y_t, Z_t, Z'_t) + \bar F(t, Y_t, Z_t, Z'_t) \leq 0,\]
but as we know $Y_{t+1}\geq \bar Y_{t+1}$, we must have $Y_{t+1}=\bar Y_{t+1}$ on $A$. Forward recursion then yields the result.
\end{proof}

Our final result shows that, in this context, any two $\SL$-expectations with the same polar sets can be expressed as solutions to $\SL$-BSDEs with respect to each other.

\begin{theorem}
 In the finite-state case, if $\F=\F_T$ for some finite $T<\infty$, for any given $\SL$-expectation $\E$, the following are equivalent.
 \begin{enumerate}[(i)]
  \item $\bar \E$ is a $\SL$-expectation absolutely continuous with respect to $\E$, in the sense that $\bar \E(I_A)=0$ for any $\E$-polar set $A$.
  \item There exists a function $F:\Omega\times\{0,...,T-1\}\times \R^{M}\times \R^{N-M-1} \to \R$ which
\begin{itemize}
 \item satisfies assumption (iv) of the comparison theorem (Theorem \ref{thm:compthm}),
 \item is sublinear, that is,  $F(\omega, t, z+\bar z, z'+\bar z') \leq F(\omega,t,z,z') + F(\omega,t,\bar z,\bar z')$ for all $z,\bar z \in \mathbb{R}^M$, $z',\bar z'\in \R^{N-M-1}$, and
 \item is positively homogenous, that is, $F(\omega, t, \lambda z, \lambda z') = \lambda F(\omega,t,z,z')$ for all $\lambda\geq 0$, all $z\in\R^M, z'\in\R^{N-M-1}$,
 \item is continuous in $z,z'$
\end{itemize}
 such that $Y_t:=\bar \E(Q|\F_t)$ is the solution to the $\SL$-BSDE with driver $F(\omega, t, z,z')$ and terminal value $Q$.
 \end{enumerate}
Furthermore, in this case $F$ is uniquely given by
\[F(z,z') = \E_t(zM_{t+1} + z'N_{t+1})-G_t(z').\]
\end{theorem}

 \begin{proof}\emph{(ii) implies (i).} We suppress the $\omega$ and $t$ arguments of $F$ for simplicity. 
Let $\bar \E_t(Q):=Y_t$, the time $t$ solution of the BSDE with terminal value $Q$. We establish the properties of Definition \ref{defn:Edefn}.
\begin{enumerate}[(i)]
\item The statement $\bar\E_t(X) \geq \bar\E_t(Y)$ whenever $X\geq Y$ is the main result of Theorem \ref{thm:compthm}.
\item That $\E_s(\E_t(Q)) = \E_s(Q)$ follows from recursivity of BSDE solutions
\item As $I_AF(z, z')= F(I_Az, I_Az')$, we see that $(I_A Y, I_A Z)$ is the solution to a BSDE with driver $F$ and terminal condition $I_A Q$. It follows that $I_A\E_t(Q) =I_A Y_t = \E_t(I_AQ)$, which implies the desired relation.
\item As $F(0,0)=0$, the solution to the BSDE with $\F_t$-measurable terminal condition $Q$ will be $(Y_s, Z_s) = (Q, 0)$ for $s\geq t$.
\item Consider the solution components $(Z,Z')$ and $(\bar Z, \bar Z')$ to the BSDE with terminal values $X$ and $Y$. Then by sublinearity of $F$, $F(Z+\bar Z, Z'+\bar Z')\geq F(Z, Z')+ F(\bar Z, \bar Z')=:\tilde F$. Applying the comparison theorem to the BSDEs with the same solution $X+Y$ but drivers $F$ and $\tilde F$, we see that $\bar \E_t(X+Y) \leq \bar\E_t(X)+\bar\E_t(Y)$.
\item It is easy to verify the equivalent statement (vi'), by multiplying the one-step equation through by $\lambda$.
\item Note that as our state-space is finite, convergence in $L^1$ is equivalent to quasi-sure convergence. We see that if $\bar\E_{t+1}(X^n)$ converges, $Z_t$ and $Z'_{t+1}$ also converge, and by continuity of $F$, it follows that $\bar\E_{t}(X)$ will also converge.
\end{enumerate}

Finally, for any $z\in\R^{M}, z'\in\R^{N-M-1}$, define
\[Y_{t+1}:= Y_t - F(\omega, t, z, z') + zM_{t+1} + z'N_{t+1} - G_t(z').\]
As $zM_{t+1} + z'N_{t+1} - G_t(z')$ is an $\E$-martingale difference term, taking an $\E_t$ expectation and rearranging gives
\[\begin{split}
F(\omega, t, z,z')&= Y_t -\E_t(Y_{t+1})\\
&=\bar\E_t(Y_{t+1}) -\E_t(Y_{t+1})\\
&=\bar\E_t(Y_{t+1} -\E_t(Y_{t+1}))\\
&=\bar\E_t(zM_{t+1}+z'N_{t+1})-G_t(z')
\end{split}\]
as desired.

\emph{1 implies 2.}
We know that, for any $0\leq t<T$, we can write
\[\bar\E_t(Q) = \bar\E_t(\bar\E_{t+1}(Q)).\]

We propose that $Y_t := \bar\E_t(Q)$ will satisfy a BSDE with driver
\begin{equation}\label{eq:FintermsofE}
F(\omega, t, z, z') := \bar\E_t(zM_{t+1}+z'N_{t+1}) - G_t(z').
\end{equation}
For any $\F_t$-measurable $Z_t, Z'_t$, as our space has only finitely many paths, $Z_t = \sum_{A_i} I_{A_i} z_i$, $Z'_t = \sum_{A_i} I_{A_i} z'_i$ for some partition $\{A_i\}$ of $\F_t$ and some constants $z_i, z'_i$. Hence by the regularity property of nonlinear expectations,
\[\begin{split}
&\quad \bar\E_t(Z_tM_{t+1}+Z'_tN_{t+1}) = \sum_i I_{A_i} \bar\E_t(Z_t M_{t+1}+Z'_tN_{t+1}) \\
&=  \bar\E_t\left(\sum_i I_{A_i}(Z_t M_{t+1}+Z'_tN_{t+1})\right) = \bar\E_t\left(\sum_i I_{A_i}z_i M_{t+1} + z'_iN_{t+1}\right)\\
&=\sum_i I_{A_i}\bar\E_t(z_i M_{t+1}+z'_i N_{t+1}),\end{split}\]
and so, for all $\F_t$-measurable $Z_t, Z'_t$,
\[F(\omega, t, Z_t(\omega), Z'_t(\omega)) = \bar\E_t(Z_tM_{t+1}+Z'_tN_{t+1})(\omega) - G_t(Z'_t(\omega)).\]

For any $Y_{t+1}=\bar\E_{t+1}(Q)$,
\[\E(Y_{t+1} - \E_t(Y_{t+1}))=0\]
Therefore we can apply Theorem \ref{thm:linearmartrepthm} to show that, for some $\F_t$-measurable $Z_t, Z'_t$,
\begin{equation}\label{eq:p21e1}
Y_{t+1} - \E_t(Y_{t+1})+ G_t(Z'_t)=Z_tM_{t+1} +Z'_tN_{t+1}.
\end{equation}

As $\mathcal{E}$ is translation invariant, with $F$ as in (\ref{eq:FintermsofE}),
\begin{equation}\label{eq:p21e2}
\begin{split}
F(\omega, t, Z_t, Z'_t) &= \bar\E_t(Z_tM_{t+1}+Z'_tN_{t+1})-G(Z'_t)\\
&=\bar\E_t(Y_{t+1} - \E_t(Y_{t+1})\\
&=Y_t - \E_t(Y_{t+1}).
\end{split}\end{equation}

Therefore, we can combine (\ref{eq:p21e1}) and (\ref{eq:p21e2}) to give
\[Y_{t+1} = Y_t - F(\omega, t, Z_{t}, Z'_t) + Z_tM_{t+1}+ Z'_tN_{t+1} -G_t(Z'_t).\]

The one-step dynamics being established, $Y_t=\bar\E_t(Q)$ satisfies the BSDE with driver $F$ by induction.

We need only to show that the driver $F$ satisfies the various conditions. Sublinearity and positive homogeneity are trivial. As $\bar\E$ satisfies monotone convergence and we have only finitely many paths, we can show that $F$ is continuous. To show that $F$ satisfies assumption (iv) of the comparison theorem, note that for any $z,z', \bar z, \bar z'$, we have
\[zM_{t+1}+z'N_{t+1} \geq  \bar zM_{t+1}+\bar z'N_{t+1} - \max_{\F_t}\{(\bar z-z)M_{t+1}+(\bar z'-z')N_{t+1}\}\]
Taking an $\bar\E_t$-expectation and rearranging gives the result.
\end{proof}

\section{Conclusion}
We have given various results regarding the theory of $\SL$-expectations in discrete time. We have studied the related $L^p$ spaces, and uniform integrability,  from which we obtain a theory of $\SL$-martingales. In this context, we have given the natural extensions of many of the standard results from classical probability theory, including optional stopping, the Doob decomposition, martingale convergence, and, in a finite-state setting, the martingale representation theorem.

We have also studied the theory of $\SL$-BSDEs in a finite-state setting, and have shown that given any $\SL$-expectation, any other $\SL$-expectation admitting the same polar sets can be expressed as the solution to a $\SL$-BSDE.

This theory of $\SL$-expectations provides the underlying discrete-time results, which can also be used when working in continuous time. For example, our optional stopping and martingale convergence results can form the basis for the corresponding results in continuous time. Unlike most results in continuous time, we do not need to restrict our attention to those random variables admitting a quasi-continuous version. This gives a greater generality to our results.

In some special cases (see \cite{Dolinsky2011}), it is known that these discrete-time operators will converge to the continuous-time $G$-expectations. We expect that more general results, including for example processes with jumps, will be developed in the future.

\bibliographystyle{plain}  
\bibliography{../RiskPapers/General}
\end{document}